\documentclass[11pt,A4paper]{article}
\usepackage{amsmath,amssymb}
\usepackage{graphicx}
\usepackage[english]{babel}
\usepackage{amsthm}
\usepackage{hyperref}
\usepackage{ulem}
\usepackage{stmaryrd} 
\usepackage[all,cmtip]{xy}
\usepackage{fullpage}
\usepackage{yfonts}
\usepackage{listings}
\usepackage{stackrel}
\usepackage{enumitem}
\usepackage{tikz}
\usepackage{tikz-cd}
\usepackage[framed,numbered]{matlab-prettifier}
\usepackage{authblk}
\tikzset{
	declare function={
		normcdf(\x,\m,\s)=1/(1 + exp(-0.07056*((\x-\m)/\s)^3 - 1.5976*(\x-\m)/\s));
	}
}

\newtheorem{theorem}{Theorem}[section] 
\newtheorem{lemma}[theorem]{Lemma}

\theoremstyle{definition}
\newtheorem{definition}[theorem]{Definition} 
\newcommand{\Q}{\mathbb{Q}}
\newcommand{\N}{\mathbb{N}}
\newcommand{\R}{\mathbb{R}}

\title{Central Limit Theorems for Martin-L\"of Random Numbers}

\author[1]{Anton Vuerinckx}
\author[1]{Yves Moreau}
\affil[1]{KU Leuven, Department of Electrical Engineering (ESAT), STADIUS Center for     
	Dynamical Systems, Signal Processing and Data Analytics
}
\date{}

\begin{document} 
	
	\maketitle
	
	\begin{abstract}
		We prove two theorems related to the Central Limit Theorem (CLT) for Martin-L\"of Random (MLR) sequences. Martin-L\"of randomness attempts to capture what it means for a sequence of bits to be ``truly random''. By contrast, CLTs do not make assertions about the behavior of a single random sequence, but only on the distributional behavior of a sequence of random variables. Semantically, we usually interpret CLTs as assertions about the collective behavior of infinitely many sequences. Yet, our intuition is that if a sequence of bits is ``truly random'', then it should provide a ``source of randomness'' for which CLT-type results should hold. We tackle this difficulty by using a sampling scheme that generates an infinite number of samples from a single binary sequence. We show that when we apply this scheme to a Martin-L\"of random sequence, the empirical moments and cumulative density functions (CDF) of these samples tend to their corresponding counterparts for the normal distribution. We also prove the well known almost sure central limit theorem (ASCLT), which provides an alternative, albeit less intuitive, answer to this question. Both results are also generalized for Schnorr random sequences. 
	\end{abstract}

	\section{Introduction}
	The concept of a random binary sequence carries different meanings in different fields. In measure-theoretic probability theory, it is often defined as a sequence of independent Bernoulli random variables (with parameter $p=1/2$). With this definition, a random sequence considers in some sense every possible sequence at once, which allows for many different kinds of propositions. Three of these in particular focus on the running sum of the sequence: The strong law of large numbers (SLLN), the law of iterated logarithm (LIL), and the central limit theorem (CLT).
	
	In algorithmic information theory however, being random is a property that an individual binary sequence can possess. Many different definitions of randomness have been proposed and compared (for some examples, see Downey and Hirschfeldt \cite[Chapter 6-7]{Downey_book}). From the beginning, the SLLN and the LIL have played a crucial role as a filter for ``bad'' notions of randomness. For example, the first definition of sequence randomness was one by von Mises \cite{von Mises} (1919). Von Mises' idea was that a sequence was random if all ``reasonably selected'' infinite subsequences satisfy the law of large numbers, but he did not provide a formal definition of what a reasonably selected subsequence actually means. Although Church \cite{Church} would later provide such formalism using the then newly developed computability theory, Ville \cite{Ville} had already shown that for any reasonable formalization, there would always be so-called von Mises-random sequences with some rare property, namely that the proportion of $1$s tends to $1/2$ much faster than expected. In particular, these sequences would not satisfy the LIL. Because of this, von Mises' definition has since been considered too weak to capture ``true'' randomness.
	
	In probability theory, both the SLLN and the LIL are statements of the form ``the set of all sequences satisfying property $P$ has measure 1''. It is therefore possible to check if a given sequence has property $P$. The classic CLT is not of this form, as it is a statement about the convergence of some cumulative distribution functions (CDFs). However, there exists several alternative versions which are applicable to single sequences, the most fundamental of which is the so-called almost sure central limit theorem. In this work, we prove the ASCLT for all Martin-L\"of sequences, and also prove an alternative version of the CLT for individual sequences by generating an infinite number of samples from a single sequence and making statements about the collective behavior of these samples. Our paper starts with the latter in Section \ref{as_section} and then proves the ASCLT for MLR sequences in Section \ref{ASCLT_section}.


	\section{Preliminaries} \label{preliminaries}
	For an easy-to-understand introduction into recursive functions, Martin-L\"of randomness and related concepts, see for example Shen \cite{Shen}.
	
	~
	
	\textbf{Notation and measure theory}
	Unless stated otherwise, a digit or sequence of digits always refers to a \textit{binary} digit $\in \{0,1\}$ or binary sequence of digits. 
	We use $x_i$ or $\omega_i$ to denote a binary digit (i.e., $\omega_i\in \{0,1\}$). Let $2^{\{n\}}$, $2^{<\N}$, and $2^{\omega}$ respectively be the set of all strings of length $n$, the set of all finite strings, and the set of all infinite sequences. 
	Any sequence $\omega$ from one of these sets has an associated real number $0.\omega_1 \omega_2\omega_3\ldots\in [0,1]$ (with all tail bits set to zero from some point on for any finite sequence). This association works essentially in both directions, as for almost all numbers in $[0,1]$ this expansion is unique. Let $x=x_1x_2x_3,\ldots x_n$ be a finite string, then $|x|$ denotes the length of this string, in this case $|x|=n$. The cylinder of $x$, $\Omega_x \subset 2^{\N}$ is the set of all infinite sequences starting with $x$ (having prefix $x$). Recall that the Lebesgue measure $\mu$ on $[0,1]$ coincides with the fair coin-tossing measure, defined by the equality $\mu(\Omega_x)=2^{-|x|}$ for all strings $x$.
	
	~
	
	\textbf{Computability theory} A countable set $A$ (of strings, numbers, $\ldots)$ is called \textit{recursively enumerable} (r.e.) if there is an algorithm which enumerates the elements of $A$. This algorithm may run indefinitely, as long as any given element of $A$ is eventually enumerated.
	
	Informally, we call a function $f$ \textit{recursive} or \textit{computable} if there is an algorithm which computes $f$, meaning that if $a$ is some valid input for $f$, then this algorithm will on input $a$ give $f(a)$ as output. This intuition is sufficient for functions from $\N$ to $\N$, but becomes inadequate when the domain or range of $f$ is an uncountable space. In the literature (see for example \cite{Gacs, Galatolo}), a function $f:A\rightarrow B$ between computable metric spaces $A,B$ is computable the inverse image of ideal balls in $Y$ are recursively enumerable opens in $X$. However, since we limit ourselves to functions $f:2^{\N} \rightarrow\R$, we use the following equivalent definition based on Turing machines (see Soare \cite{Soare_book}): A function $f:2^{\N} \rightarrow\R$ if there is a Turing machine $A$ which takes an oracle $\omega\in 2^{\N}$ and $\epsilon \in \Q^+$ as input and outputs $r\in \Q$ such that $|f(\omega)-r |<\epsilon$.
	
	For a given sequence $(O_{n})$ of computable objects (numbers/functions), we can ask whether this sequence is \textit{uniformly computable} (in $n$), which intuitively means there is some algorithm $A$ such that $A(n)=O_n$. For example, a sequence of reals $(a_n)$ is uniformly computable in $n$ if there is an Turing machine which on input $n$ and rational $\epsilon>0$, outputs $r\in \Q$ such that $|r-a_n|<\epsilon$. 
	
	Lastly, let $a_{1:\infty}$ be a sequence of numbers which converges to some limit $a$, then we say that $a_{1:\infty}$ \textit{converges effectively} to $a$ if there is some algorithm which on any rational input $\epsilon>0$, outputs an $N\in \mathbb{N}$ such that $|a_n-a|<\epsilon$ for any $n\geq N$. Similarly, a sequence $a_{1:\infty}$ \textit{diverges effectively} to $+\infty$ if there is an algorithm which outputs $N$ on input $M$ such that $a_n>M$ for all $n\geq N$.
	
	~
	
	\textbf{Algorithmic randomness} We will consider both Martin-L\"of randomness and the slightly weaker Schnorr randomness: A set $U=\Omega_{x_1}\cup \Omega_{x_2}\cup \ldots$ is \textit{effectively open} if the set of strings $\{x_1,x_2,\ldots\}$ is recursively enumerable. 
	A sequence of effectively open sets $U_{1:\infty}$ is called uniformly r.e. if there is a recursive function $g$ such that $g(i)$ outputs an enumeration $\{x_1,x_2,\ldots\}$ where $U_i=\Omega_{x_1}\cup \Omega_{x_2}\cup \ldots$.
	A Martin-L\"of test is then a uniformly r.e. sequence $U_{1:\infty}$ such that $\lambda(U_i)\leq 2^{-i}$. A Schnorr test is a Martin-L\"of test with the added requirement that the measures $\mu(U_n)$ are uniformly computable in $n$. Finally, a sequence is Martin-L\"of random (MLR), resp. Schnorr random, if it is not contained in the intersection $\bigcap_i U_i$ of any Martin-L\"of, resp. Schnorr, test. 
	\section{An alternative central limit theorem}\label{as_section}
	\subsection{Main theorem and proof}
	
	In this section, we explain and prove an alternative version of the CLT, where infinitely many samples are generated from a single infinite sequence which at infinity resemble samples from the normal distribution. 
	
	Naturally, we will work with functions between the Cantor space and $\R$. Since the fair coin-tossing measure is a probability measure on $2^\N$, these functions can be seen as random variables. In order to avoid confusion when switching between the viewpoints of random variables and functions, we will refer to such function as Sequence-Based Variables, or SBVs for short:
	\begin{definition} A \textit{sequence-based variable} $X$ is a computable function $2^\omega\rightarrow \R$. A pair of computable functions $(I,f)$ is called a \textit{representation} of $X$ if $X=f\circ I$, where $I:2^\N\rightarrow 2^{\{N\}}: (\omega_1,\omega_2,\omega_3,\ldots)\mapsto (\omega_{i_1},\omega_{i_2},\ldots)$, ($N\in \mathbb{N}$ or $\{N\}=\N$) and $f:2^{\{N\}}\rightarrow \mathbb{R}$. The function $I$ in such a representation is referred to as a \textit{selection function} and set of the indices $\{i_1,i_2,\ldots\}$ as the \textit{selected indices}.
	\end{definition}
	A first example of an SBV $X$ is simply the sum of the first $n$ digits of a sequence:
	\[  \underbrace{\omega_1\qquad \omega_2\qquad \omega_3 \qquad \ldots \qquad \omega_n}_{X=\omega_1+\omega_2+\ldots+\omega_n}\qquad \omega_{n+1} \qquad \omega_{n+2} \qquad \ldots. \]
	Any SBV has of course many different representations (we can always add extra inputs to $f$ or change the order of the selected indices). However, one obvious representation for our example above is $(I,f)$ where $I$ simply selects the first $n$ digits in order and $f$ is the function of arity $n$ that adds its inputs. 
	
	As mentioned before, SBVs inherit all concepts related to random variables by applying them to a sequence of i.i.d. Bernoulli RV (as an example, the SBV defined above inherits a binomial distribution $B(n,1/2)$). Hence, we can talk about independent SBVs, the expectation of an SBV, etc. We also have the following properties:
	\begin{itemize}[noitemsep]
		\item If two SBVs have representations such that sets of the selected indices are disjoint, then these SBVs are independent. 
		\item If $X$ is an SBV with representation $(I,f)$ and $f$ has finite arity $n$, then
		\[\mathbb{E}\left[X\right]=\sum_{(\omega_{1},\omega_{2},\ldots,\omega_{n})\in 2^{\{n\}} } \frac{1}{2^n}f(\omega_{1},\omega_{2},\ldots,\omega_{n}), \]
		which demonstrates that changing the selection function does not change the expectation.
	\end{itemize}
	For an SBV $X$, we will use $X$ to denote both the true SBV (the function) as well as the output of applying $X$ to a specific sequence $\omega$. It will be clear from context whether $X$ must be seen as a function or an output of that function. 
	
	~
	
	Now, we look at appropriate SBVs to generate our samples. Recall the CLT (see for example Feller \cite[p244]{Feller_book}): Let $(B_{n})$ be a sequence of iid (independent and identically distributed) random variables with mean $\mu$ and variance $\sigma^2$ and let $\Phi$ be the CDF of the standard normal distribution (mean $0$ and variance $1$). Letting $S_n=B_1+B_2+\ldots+B_n$, then for any $a\in \mathbb{R}$, 
	\[ P\left( \frac{S_n-n\mu}{\sigma\sqrt{n}}<a \right)\rightarrow \Phi(a),  \] 
	often written as $\frac{S_n-n\mu}{\sigma\sqrt{n}}\xrightarrow{D} N(0,1)$.
	
	In case the $B_i$ are Bernoulli random variables with parameter $p=1/2$, the statement becomes $\frac{2S_n-n}{\sqrt{n}}\xrightarrow{D} N(0,1)$. Based on this, we use the following set of SBVs to generate our samples\footnote{There are of course multiple alternatives to this particular set of SBVs, some of which are briefly discussed in Section \ref{variations_section}.}:
	\begin{align}\label{sampling_scheme}
		\underbrace{\omega_1}_{ X_1=\frac{2\omega_1-1}{\sqrt{1}} }\quad\underbrace{\omega_2\quad \omega_3}_{ X_2=\frac{2(\omega_2+\omega_3)-2}{\sqrt{2}} }\quad \underbrace{\omega_4\quad \omega_5\quad \omega_6}_{ X_3=\frac{2(\omega_4+\omega_5+\omega_6)-3}{\sqrt{3}} }\qquad \underbrace{ \omega_7 \quad \omega_8\quad \omega_9\quad \omega_{10}}_{X_4=\frac{2(\omega_7+\omega_8+\omega_9+\omega_{10})-4}{\sqrt{4}}}\qquad \underbrace{\omega_{11} \quad\ldots}_{X_5=\ldots },
	\end{align}
	always using $n$ digits in the definition of $X_n$. We will act as if these $X_i$ are our random samples from some distribution. Let $\widehat{D}_k$ be the empirical distribution of the first $k$ samples $X_1,\ldots, X_k$ (meaning $\widehat{D}_k$ is a RV whose CDF is the empirical distribution function of $\{X_1,\ldots,X_k \}$). We will show that if $\omega$ is MLR, then the empirical moments of these distributions converge to the corresponding moments of the normal distribution. Formally for any MLR sequence $\omega$, we will show that
	\begin{align}\label{moments_conv}
		\mathcal{E}\left[ \widehat{D}_k^m \right]=\frac{X_1^m+X_2^m+\ldots+X_k^m}{k} \rightarrow \nu_{m}, \qquad \text{as }k\rightarrow\infty, \text{ for all } m\in \mathbb{N}_0,
	\end{align}
	where $\nu_{m}$ denotes the $m$-th moment of the standard normal distribution (see Papoulis \cite[p148]{Papoulis_book}):
	\begin{align}\label{normalmoments}
		\nu_m=\begin{cases}
			0& \text{if } m \text{ is odd},\\
			\displaystyle 
			(m-1)!! & \text{if } m \text{ is even\footnotemark}.
		\end{cases}
	\end{align}
	\footnotetext{$(m-1)!!$ denotes the double factorial: $(m-1)!!=(m-1)\cdot (m-3)\cdot (m-5)\cdot \ldots \cdot 3\cdot 1$.}
	
	We use the notation $\mathcal{E}$ instead of $\mathbb{E}$ to stress that $\mathcal{E}\left[ \widehat{D}_k^m \right]$ is an SBV (its value depends on $\omega$), not a number. 
	
	As shown later, if (\ref{moments_conv}) holds for a given sequence, then the corresponding empirical distribution functions $\widehat{F}_k$ defined as
	\begin{align}\label{emp_dist_func}
		\widehat{F}_k(t)= \frac{\pmb{1}_{\left\{X_1\leq t \right\}} +\pmb{1}_{\left\{X_2\leq t \right\}}+\ldots +\pmb{1}_{\left\{X_k\leq t \right\}}}{k},\qquad t\in \mathbb{R}
	\end{align} 
	will converge uniformly to the CDF of the standard normal distribution. Note that for any $k$ and $t$, $\widehat{F}_k(t)$ is an SBV. These results are summarized in Theorem \ref{main_result}. 
	
	~
	
	To prove the convergence of the moments (\ref{moments_conv}), we use an adapted version of the SLLN due to Kolmogorov. 
	
	\begin{theorem}[Kolmogorov's SLLN for MLR sequences]\label{kolmogorov}
		Let $(X_n)$ be a sequence of independent sequence-based variables, uniformly in $n$. Let $\mu_n$ and $\sigma_n^2$ respectively be the average and variance of $X_n$\footnote{These exist and are uniformly computable: Since $2^{\N}$ is compact and $X_n$ are computable, it is an easy to show that $\int X_n(\omega)\, d\omega$ and $\int X_n^2(\omega)\, d\omega$ are uniformly computable in $n$.}
		Assume the $X_n$ have finite averages $\mu_n$ and variances $\sigma_n^2$. Also assume the following holds
		\begin{enumerate}
			\item The sequences $(\mu_n)$ converges effectively to some $\mu \in\R$.
			\item $\sum_k k^{-2}\sigma_k^2$ is a finite, computable number.
		\end{enumerate}
		Then for all Martin-L\"of random sequences
		\begin{align}\label{converge}
			\bar{X}_n \rightarrow \mu \qquad \text{as }n\rightarrow \infty,
		\end{align}
		where $\bar{X}_n=\frac{1}{n}\sum_{i\leq n}X_i$.
	\end{theorem}
	The proof of this theorem can be found at the end of this section. 
	The proof of the original statement by Kolmogorov can be found in Sen \cite[p67]{Sen}, where of course conditions \textit{1} and \textit{2} were not present. These conditions are necessary to make an effective version of the proof in \cite{Sen}, but are perhaps not required for the theorem to hold. This paper does not make any attempts to answer whether or not they are indeed required.
	
	Another minor difference: Usually, the condition $\mu_n\rightarrow \mu$ is not present and the theorem simply states that
	\begin{align}\label{converge_original}
		\bar{X}_n-\bar{\mu}_n \rightarrow 0 \qquad \text{a.s. as }n\rightarrow \infty,
	\end{align}
	where $\bar{\mu}_n=\frac{1}{n}\sum_{i\leq n}\mu_i$. It is an easy exercise to show that (\ref{converge}) and (\ref{converge_original}) are equivalent when $\mu_n \rightarrow \mu$ and that the convergence of the $(\bar{\mu}_n)$ is effective when the $(\mu_n)$ converge effectively.
	
	~
	
	The conditions  are necessary to make an effective version of the proof by Shen. It might be possible to weaken these conditions.
	
	~
	
	We will, for any $m\in \mathbb{N}_0$, apply Theorem \ref{kolmogorov} to the sequence $X_1^m, X_2^m,\ldots$ where the $X_i$ are those defined in (\ref{sampling_scheme}) to show that (\ref{moments_conv}) holds. However, before we can apply the theorem, we need some knowledge on the expectation and variance of the $X_i^m$.
	
	Note that the SBVs in (\ref{sampling_scheme}) can also be defined by first transforming the original sequence $\omega$ into a Rademacher sequence $r$ by the equation $r_i=2\omega_i-1$ (replacing any $0$s with $-1$s) and then defining 
	\begin{align}\label{alternative}
		\underbrace{r_1}_{ X_1=\frac{r_1}{\sqrt{1}} }\qquad\underbrace{r_2\qquad r_3}_{ X_2=\frac{r_2+r_3}{\sqrt{2}} }\qquad \underbrace{r_4\qquad r_5\qquad r_6}_{ X_3=\frac{r_4+r_5+r_6}{\sqrt{3}} }\qquad \underbrace{ r_7 \qquad r_8\qquad r_9\qquad r_{10}}_{X_4=\frac{r_7+r_8+r_9+r_{10}}{\sqrt{4}}}\qquad \underbrace{r_{11} \qquad\ldots}_{X_5=\ldots }.
	\end{align}
	
	Hence, we use the following lemma, which will allow us to use Theorem \ref{kolmogorov}.
	
	\begin{lemma}\label{rad_lemma}
		Let $r_i$ be iid Rademacher distributed and define $S_n=r_1+r_2+\ldots +r_n$. Then for any fixed $m\geq 1$,
		\begin{align*}
			\mathbb{E}[S_n^m]=\begin{cases}
				0& \text{if } m \text{ is odd},\\
				\displaystyle 
				(m-1)!!\cdot n^{m/2}+\mathcal{O}\left(n^{\frac{m}{2}-1}\right) & \text{if } m \text{ is even}.
			\end{cases}
		\end{align*}
	\end{lemma}
	
	The proof of this lemma is given at the end of this section, as we first demonstrate how it applies to our situation. Letting $X_i$ as always denote the SBVs in (\ref{sampling_scheme}) (and alternatively (\ref{alternative})), 
	\begin{align}\label{expectationSBV}
		\mathbb{E}\left[X_n^m\right]=\mathbb{E}\left[\left(\frac{r_1+\ldots+r_n}{\sqrt{n}}\right)^m\right]=\begin{cases}
			0& \text{if } m \text{ is odd},\\
			\displaystyle 
			(m-1)!!+\mathcal{O}\left(\frac{1}{n}\right) & \text{if } m \text{ is even}.
		\end{cases}
	\end{align} 
	
	Hence, we can see that these expectations converge to the moments of the standard normal distribution. This fact alone can be proven much more easily, but (\ref{expectationSBV}) also gives a computable upper bound on the error, which which implies that the convergences are effective. Using this expression, it can also be seen that the variance $\text{Var}\,X_n^m$ converges to $(2m-1)!!$ if $m$ is odd and to $(2m-1)!!-[(m-1)!!]^2$ if $m$ is even. In particular, the variances remain bounded. This implies that $\sum_k k^{-2}\text{Var}\,X_k^m <\infty$ and that this sum is computable, since we have an upper bound for the size of the tail. Thus, we may apply Theorem \ref{kolmogorov} to the sequence $(X_{n}^m)$, which finally implies that (\ref{moments_conv}) holds for all MLR sequences. 
	
	~
	
	Next, consider one of the sequences $\omega$ for which (\ref{moments_conv}) holds. We show that for such a sequence the empirical distribution function converges uniformly to $\Phi$. This result follows directly from the following theorem. 
	
	\begin{theorem}\label{method_of_moments}
		The following statements hold
		\begin{enumerate}[noitemsep]
			\item The standard normal distribution is completely determined by its moments, meaning that if some r.v. $X$ has the same moments as the standard normal distribution, then $X\sim N(0,1)$. 
			\item  Suppose that the distribution of $X$ is determined by its
			moments, that the $X_n$ have moments of all orders, and that $\lim_n \mathbb{E}[X_n^m] =
			\mathbb{E}[X^m]$ for $m = 1, 2, \ldots$. Then $X_n \xrightarrow{D} X$.
			\item Let $X_n,X$ be random variables such that $X_n \xrightarrow{D} X$ and let $F_n,F$ respectively denote their CDF. If $F$ is continuous, then
			\[ \sup_{x}|F_n(x)-F(x)|\rightarrow 0 \qquad \text{as }n\rightarrow \infty. \]
		\end{enumerate}
	\end{theorem}
	
	Indeed, for such a fixed sequence, the empirical distributions $\widehat{D}_n$ form a sequence of random variables whose moments converge to those of the normal distribution. Hence, we can apply Theorem \ref{method_of_moments} to find that $\widehat{F}_n\rightarrow\Phi$ uniformly for all MLR sequences $\omega$, where 
	$\widehat{F}_n(t)$ is understood as the output of the SBV $\widehat{F}_n(t)$ on input $\omega$.
	The proof of Statements \textit{1} and \textit{2} can be found in Billingsley \cite[Ex. 30.1 \& Thm 30.2]{Billingsley_book}. Statement \textit{3} is proven in Chow \cite[p260]{Chow_book}. An alternative proof of the pointwise convergence of $\widehat{F}_n(t)$ can be found by applying Theorem \ref{kolmogorov} to the sequence $\pmb{1}_{\left\{X_k\leq t \right\}}$. We summarize these results in the following theorem.
	
	\begin{theorem}\label{main_result}
		Let $\omega$ be a Martin-L\"of random sequence and define the sequence $(X_n)$ as in \eqref{sampling_scheme}.
		Let $\widehat{D}_k$ be the random variable assigning probability $1/k$ to any outcome $X_i,i=1,\ldots,k$ and let $\widehat{F}_k$ be the corresponding CDF. Then $\widehat{D}_k$ approximates the normal distribution in the following ways:
		\begin{enumerate}[noitemsep]
			\item  Let $\nu_m$ denote the $m$-th moment of the standard normal distribution. Then for every $m\in \N_0$:
			\[ \mathbb{E}\left[\widehat{D}_k^m\right]=\frac{X_1^m+X_2^m+\ldots+X_k^m}{k}\rightarrow \nu_m \qquad \text{ as }k\rightarrow \infty. \]
			In other words, all moments of $\widehat{D}_k$ converge to the corresponding moments of the standard normal distribution.
			\item The CDFs $\widehat{F}_k$ converge uniformly to the CDF $\Phi$ of the standard normal distribution, i.e.
			\begin{align*}
				\sup_{t}|\widehat{F}_k(t)-\Phi(t)|\rightarrow 0 \qquad \text{as }k\rightarrow \infty.
			\end{align*}
		\end{enumerate}
	\end{theorem}
	
		An alternative way to prove Theorem 3.5 is to consider the sequence $(X_k)$ as the image of a map $f:2^\N \rightarrow \R^\N:\omega\mapsto (X_k)$. If we then endow the space of real sequences $\R^\N$ with the push-forward measure $\mu_f$ defined as $\mu_f(A)=\mu(f^{-1}(A))$ (with $\mu$ the Lebesgue measure on $2^\N$), we can apply \cite[Theorem 3.9]{Hoyrup2} which says that if $x$ is ML-random, so is $f(x)$. Theorem \ref{main_result} then reduces to proving that a ML-random sequence $(X_k)$ in this $\mathbb{R}^\N$ has the property that $\mathbb{E}\left[\widehat{D}_k^m\right]$ as defined above converges to $\nu_m$ for all $m\in \mathbb{N}$. Because of the nature of the measure $\mu_f$, this would likely require proving a statement similar to Theorem \ref{kolmogorov}.

	~
	
	Theorem \ref{kolmogorov} and, therefore, Theorem \ref{main_result} also hold for Schnorr random sequences: When looking at the proof of Theorem \ref{kolmogorov} (see below), the measures of the $K_{M}^{\epsilon}$ are computable since 
	\[ K_{M}^{\epsilon}= \left\{ \omega \mid \sup\limits_{ M\leq k\leq N  }\, |\bar{X}_k-\mu|>\epsilon \right\}\cup\left\{ \omega \mid \sup\limits_{ k\geq N  }\, |\bar{X}_k-\mu|>\epsilon \right\} \]
	for any $N\geq M$, which splits $K_{M}^{\epsilon}$ into a set described by a finite amount of strings, and a set whose measure effectively tends to $0$ as $N\rightarrow \infty$. The measures of the $U_n=\cup_{l=1}^\infty K_{M(n,l)}^{2^{-l}}$ are then also uniformly computable since the measure of the tail $\cup_{l=L}^\infty K_{M(n,l)}^{2^{-l}}$ is bounded by $2^{-n-L+1}$.
	
	~
	
	We end this section with the proofs of Theorem \ref{kolmogorov} and Lemma \ref{rad_lemma}.
	\begin{proof}[Proof of Theorem \ref{kolmogorov}, based on Sen \cite{Sen}]
		Let $T_k=\sum_{n\leq k}(X_n-\mu_n)$ and $T_0=0$. We start from the following inequality (see Sen \cite[Eq 2.3.65]{Sen}):
		\begin{align}\label{derivedineq}
			P\left(\max\limits_{M\leq k\leq N}\,\frac{1}{k} |T_k|>t\right)\leq \frac{1}{t^2}\left[ \frac{1}{M^2}\left( \sum_{k=1}^{M}\sigma_k^2 \right)+\sum_{k=M+1}^{N} k^{-2}\sigma_k^2\right] .	
		\end{align}
		Let for any $k\geq 1$ $D_k=\sum_{n\geq k}n^{-2}\sigma_n^2$. Note that by assumption, $D_1=\sum_{n}n^{-2}\sigma_n^2<\infty$ and hence, $(D_k)$ is a bounded decreasing sequence with $\lim_kD_k=0$ (since $D_1$ is computable, this convergence is effective). Also,
		\begin{align*}
			\frac{1}{M^2}\sum_{k=1}^{M}\sigma_k^2&=\frac{1}{M^2}\sum_{k=1}^{M}k^2[D_k-D_{k+1}]\leq \frac{1}{M^2}\sum_{k=1}^{M}(2k-1)D_k.
		\end{align*}
		Since $M^{-2}\sum_{k=1}^{M}(2k-1)\rightarrow 1$ and $D_k\rightarrow 0$, it is easy to show that the rhs (right-hand side) converges to $0$ (and does so effectively since $D_k$ also converges effectively to $0$). Hence, the lhs converges effectively to $0$. Finally, we can show that $\bar{X}_k\rightarrow \mu$ for all MLR sequences. Fix $\epsilon>0$ and note that by the triangle inequality, $|\bar{X}_k-\mu|>\epsilon$ implies that either $|\bar{X}_k-\bar{\mu}_k|> \epsilon/2$ or $|\bar{\mu}_k-\mu|> \epsilon/2$. As mentioned before, it is easy to show that $\bar{\mu}_k\rightarrow\mu$ and that this convergence is effective if and only if the convergence of $(\mu_n)$ is effective. Since the latter is part of our assumptions, we can for any $\epsilon$, find $M_0$ such that $|\bar{\mu}_k-\mu|< \epsilon/2$ for all $k\geq M_0$. For such a $k$, we have
		\[ P\left( |\bar{X}_k-\mu|> \epsilon \right)\leq P\left( |\bar{X}_k-\bar{\mu}_k|> \frac{\epsilon}{2} \right)=P\left(\frac{1}{k} |T_k|> \frac{\epsilon}{2}  \right).  \]
		Using the previous argument simultaneously for all $k$ between $M\geq M_0$ and some $N$, we find that
		\begin{align*}
			P\left( \max\limits_{M\leq k\leq N}\,|\bar{X}_k-\mu|> \epsilon \right) \leq  P\left(\max\limits_{M\leq k\leq N}\,\frac{1}{k} |T_k|>\frac{\epsilon}{2}\right)\leq \frac{4}{\epsilon^2}\left[ \frac{1}{M^2}\left( \sum_{k=1}^{M}\sigma_k^2 \right)+\sum_{k=M+1}^{N} k^{-2}\sigma_k^2\right].
		\end{align*}
		
		Letting $N$ tend to infinity and keeping $M$ fixed gives
		\begin{align}\label{K_eps}
			P\left(\sup\limits_{k\geq M}\, |\bar{X}_k-\mu|>\epsilon\right)\leq \frac{4}{\epsilon^2}\left[ \frac{1}{M^2}\left( \sum_{k=1}^{M}\sigma_k^2 \right)+D_{M+1}\right].
		\end{align}
		For any $\epsilon>0$ the rhs effectively tends to $0$ as $M\rightarrow \infty$ and hence, the r.e. sets $K_M^{\epsilon}=\{ \omega \in 2^{\N}\mid \sup\limits_{k\geq M}\, |\bar{X}_k-\mu|>\epsilon \}$
		allow us to define the Martin-L\"of test: let $M(n,l)$ be a computable function such that
		\begin{align*}
			\mu\left(K_{M(n,l)}^{2^{-l}}\right)=P\left(\sup\limits_{k\geq M(n,l)}\, |\bar{X}_k-\mu|>2^{-l}\right)<2^{-(n+l)}
		\end{align*} 
		and define $U_n=\cup_{l=1}^\infty K_{M(n,l)}^{2^{-l}}$. The $U_n$ are uniformly r.e. since the $K_{M(n,l)}^{2^{-l}}$ are uniformly r.e. in $n,l$ and by construction, $\mu(U_n)<2^{-n}$. Hence, the $U_n$ define a Martin-L\"of test whose intersection contains all sequences for which $\bar{X}_n \nrightarrow \mu$.
	\end{proof}

	\begin{proof}[Proof of Lemma \ref{rad_lemma}]
		The case where $m$ is odd is trivial: Note that the distribution of $S_n$ (and therefore $S_n^m$) is completely symmetric around $0$. Hence $\mathbb{E}[S_n^m]=0$. 
		
		~
		
		Now for the case where $m$ is even: Using the multinomial theorem (Spiegel \cite[p3]{Spiegel_book}), we find
		\begin{align*}
			S_n^m=&(r_1+r_2+\ldots+r_n)^m\\
			=& \sum_{k_1+k_2+\ldots+k_n=m} \binom{m}{k_1,k_2,\ldots,k_n}r_1^{k_1}r_2^{k_2}\ldots r_n^{k_n},
		\end{align*}
		where \[\binom{m}{k_1,k_2,\ldots,k_n}=\frac{m!}{k_1!\cdot k_2!\cdot \ldots \cdot k_n!}. \]
		Taking expectation, we find
		\begin{align*}
			\mathbb{E}[S_n^m]=\sum_{k_1+k_2+\ldots+k_n=m} \binom{m}{k_1,k_2,\ldots,k_n}\mathbb{E}[r_1^{k_1}]\mathbb{E}[r_2^{k_2}]\ldots \mathbb{E}[r_n^{k_n}].
		\end{align*}
		Since $\mathbb{E}[r_i^{k_i}]$ is $0$ if $k_i$ is odd and $1$ if $k_i$ is even, this simplifies to
		\begin{align*}
			\mathbb{E}[S_n^m]=\sum_{\substack{k_1+k_2+\ldots+k_n=m\\\text{all }k_i\text{ even}}} \binom{m}{k_1,k_2,\ldots,k_n}
		\end{align*}
		Simplifying the summation such that we only sum over distinct partitions $(k_1,k_2,\ldots,k_l)$ of $m$ yields 
		\begin{align*}
			\mathbb{E}[S_n^m]=&\sum_{\substack{p=(k_1,k_2,\ldots,k_l)\\\text{partition of }m, \\\text{all }k_i\text{ even}
			}} \binom{m}{k_1,k_2,\ldots,k_l}K_{p}^{(n)},
		\end{align*}
		with
		\begin{align*}
			K_p^{(n)}=K_{(k_1,k_2,\ldots,k_l)}^{(n)}=\begin{aligned}
				&\text{ The number of ways we can write }k'_1+k'_2+\ldots+k'_n=m\\
				& \text{ where the non-zero terms are exactly the }k_i.
			\end{aligned}
		\end{align*}
		While an explicit formula for $K_p^{(n)}$ is quite cumbersome, we only give the following bounds, which suffice for the coming discussion:
		\[  \binom{n}{l}\leq K_{(k_1,k_2,\ldots,k_l)}^{(n)} \leq \frac{n!}{(n-l)!}<n^l. \]
		The lower (resp. upper) bound is achieved by assuming that all the $k_i$ are the same (resp. different). In particular, $K_{p}^{(n)}=\mathcal{O}(n^l)$ as $n\rightarrow \infty$. 
		
		Note that $(k_1,k_2,\ldots,k_l)$ is a partition of $m$ with even coefficients if and only if  $\left(\frac{k_1}{2},\frac{k_2}{2},\ldots,\frac{k_l}{2}\right)$ is a partition of $\frac{m}{2}$. Hence
		\begin{align*}
			\mathbb{E}[S_n^m]=&\sum_{\substack{p=(k_1,k_2,\ldots,k_l)\\\text{partition of }\frac{m}{2}}} \binom{m}{2k_1,2k_2,\ldots,2k_l}K_{2p}^{(n)}=\sum_{l=1}^{m/2} \sum_{\substack{p=(k_1,k_2,\ldots,k_l)\\\text{partition of }\frac{m}{2} \\\text{of length }l}} \binom{m}{2k_1,2k_2,\ldots,2k_l}K_{2p}^{(n)},
		\end{align*}
		
		Note that there is only 1 partition of length $l=m/2$, namely $(1,\ldots,1)$. In that case, $K_{2p}^{(n)}=\binom{n}{\frac{m}{2}}$. Hence,
		\begin{align*}
			\mathbb{E}[S_n^m]=& \sum_{l=1}^{m/2} \sum_{\substack{p=(k_1,k_2,\ldots,k_l)\\\text{partition of }\frac{m}{2} \\\text{of length }l}} \binom{m}{2k_1,2k_2,\ldots,2k_l}K_{2p}^{(n)}\\
			=& \binom{n}{\frac{m}{2}}\frac{m!}{2^{m/2}} + \sum_{l=1}^{m/2-1} \sum_{\substack{p=(k_1,k_2,\ldots,k_l)\\\text{partition of }\frac{m}{2} \\\text{of length }l}} \binom{m}{2k_1,2k_2,\ldots,2k_l}K_{2p}^{(n)}\\
			=& \frac{m!}{\left(\frac{m}{2}\right)!2^{m/2} }n^{m/2}+\mathcal{O}\left(n^{\frac{m}{2}-1}\right)\\
			=& (m-1)!!\cdot n^{m/2}+\mathcal{O}\left(n^{\frac{m}{2}-1}\right).
		\end{align*}
	\end{proof}

	\subsection{Alternative versions}\label{variations_section}
	
	Recall that our way of defining the sampling scheme (\ref{sampling_scheme}) was mostly arbitrary. As the only requirement was that the samples resembled $\frac{2S_n-n}{\sqrt{n}}$ for some large $n$, nothing forced us to require that $X_k$ uses exactly $k$ digits. A slightly more general approach is to let $X_k$ use $n(k)$ digits for some computable function $n:\mathbb{N} \mapsto \mathbb{N}$. For example, if $n(k)=2k$, the SBVs become
	\begin{align*}
		\underbrace{\omega_1\quad \omega_2}_{ X_1=\frac{2(\omega_1+\omega_2)-2}{\sqrt{2}} }\qquad\underbrace{ \omega_3 \quad \omega_4\quad \omega_5\quad \omega_6}_{ X_2=\frac{2(\omega_3+\omega_4+ \omega_5+\omega_6)-4}{\sqrt{4}} }\qquad\underbrace{\omega_7\quad \omega_8 \quad \omega_9\quad \omega_{10}\quad \omega_{11} \quad \omega_{12}}_{ X_3=\frac{2(\omega_7+ \omega_8 + \omega_9+ \omega_{10}+ \omega_{11} + \omega_{12})-6}{\sqrt{6}} }\qquad \underbrace{\omega_{13}\quad \ldots}_{ X_4=\ldots }.
	\end{align*}
	
	In this section, we consider two cases: $n(k)\rightarrow \infty$ effectively and $n(k)=N$ for some fixed $N\in \mathbb{N}$.
	
	~
	
	\underline{$n(k)\rightarrow  \infty$ effectively}: As a generalization of (\ref{expectationSBV}), we have
	\begin{align*}
		\mathbb{E}\left[X_k^m\right]=\begin{cases}
			0& \text{if } m \text{ is odd},\\
			\displaystyle 
			(m-1)!!+\mathcal{O}\left(\frac{1}{n(k)}\right) & \text{if } m \text{ is even}.
		\end{cases}
	\end{align*} 
	Hence, the averages (and variances) still converge to the same values. As nothing truly changes, Theorem \ref{main_result} still holds for these SBVs. 
	
	~
	
	$\underline{n(k)=N}:$ As an example, if $N=4$ the SBVs become
	\begin{align*}
		\underbrace{\omega_1\quad \omega_2\quad \omega_3\quad \omega_4}_{ X_1=\frac{2(\omega_1+\omega_2+\omega_3+ \omega_4)-4}{\sqrt{4}} }\qquad\underbrace{ \omega_5\quad \omega_6\quad \omega_7\quad \omega_8}_{ X_2=\frac{2(\omega_5+ \omega_6+ \omega_7+ \omega_8)-4}{\sqrt{4}} }\qquad\underbrace{\omega_9\quad \omega_{10}\quad \omega_{11} \quad \omega_{12}}_{ X_3=\frac{2(\omega_9+ \omega_{10}+ \omega_{11} + \omega_{12})-4}{\sqrt{4}} }\qquad \underbrace{\omega_{13}\quad \ldots}_{ X_4=\ldots }.
	\end{align*}
	Of course, the corresponding moments $\mathbb{E}\left[\widehat{D}_k^m \right]$ and distribution function $\widehat{F}_k$ no longer converge to those of the normal distribution, but instead to those of a normalized binomial distribution. However, as $N$ tends to infinity, these moments and distribution function themselves tend to those of the normal distribution function. Hence, we have for all MLR sequences:
	\[ \lim\limits_{N\rightarrow \infty}\lim\limits_{k\rightarrow \infty} \frac{X_1^m+X_2^m+\ldots+X_k^m}{k} \rightarrow \nu_{m},\qquad \lim\limits_{N\rightarrow \infty}\lim\limits_{k\rightarrow \infty} \widehat{F}_k(t)=\Phi(t) \quad t\in \R. \]

	Note that the only condition that is required for this version of the CLT, is that the original sequence $\omega$ is normal. Schnorr random sequences are of course all normal, see Shen \cite[Thm 167 \& 168]{Shen}.
	
	~
	
		Sampling scheme (\ref{sampling_scheme}) can also be generalized by introducing an in-between function $f$ in the following way:
		\begin{align}\label{sampling_scheme2}
			\underbrace{\omega_1 \ldots \omega_k}_{ X_1=\frac{f(\omega_{1:k})}{\sqrt{1}} }\quad\underbrace{\omega_{k+1}\ldots \omega_{2k}\quad \omega_{2k+1}\ldots \omega_{3k}}_{ X_2=\frac{f(\omega_{(k+1):2k})+f(\omega_{(2k+1):3k})}{\sqrt{2}} }\quad \underbrace{\omega_{3k+1}\ldots \omega_{4k}\quad \omega_{4k+1}\ldots \omega_{5k}  \quad \omega_{5k+1}\ldots \omega_{6k}}_{ X_3=\frac{f(\omega_{(3k+1):4k})+f(\omega_{(4k+1):5k})+f(\omega_{(5k+1):6k})}{\sqrt{3}} }
		\end{align}
		with the added condition that when $f$ is applied to standard Bernoulli random variables $\mathbb{E}[f]=0$ and $\mathbb{E}[f^2]=1$. The natural question is then to ask what conditions can be placed on $f$ such that Theorem \ref{main_result} still holds. Another question is whether the theorem still holds if the definition of $X_k$ uses slightly overlapping digits of $\omega$. We leave these questions for future work.

	\section{The almost sure central limit theorem}\label{ASCLT_section}
	
	In this section, the $X_k$ no longer denote the SBVs in \eqref{sampling_scheme}, but some other SBVs such that $E[X_k]=0$ and $E[X_k^2]=1$ (in particular, $X_k=\omega_k$).
	
	We will show that the classic ASCLT (see Brosamler \cite{Brosamler})
	\begin{align}\label{ASCLT}
		\lim_{n\rightarrow\infty }\frac{1}{\log n}\sum_{k=1}^n\frac{1}{k} I\left\{\frac{S_k}{\sqrt{k}}\leq x \right\}\rightarrow \Phi(x)\quad \text{ for all }x\in \R 
	\end{align}
	holds for all Martin-L\"of (and even Schnorr sequences). Our proof is a based on the one presented by Jonsson \cite{Jonsson}. Since we only show this basic form, we can substantially shorten that proof (apart from adding some computability restrictions).
	
	~
	
	We start this section with three preliminary lemmas. Although slightly restated in the context of MLR sequences, the proofs of these lemmas are identical to those found in \cite{Jonsson} (respectively Theorem 2.4, Theorem 2.25 \& Lemma 3.2). 
	
	\begin{lemma}\label{portmanteau}
		Let $(d_k)$ be a sequence of positive uniformly computable real numbers with $D_n=\sum_{k\leq n}d_k$. Also, let $(X_k)$ be a sequence of SBVs and let $G$ be a probability measure on $\R$ with $C_G$ its continuity points, then the following three conditions are equivalent:
		\begin{enumerate}
			\item \label{cond:first}  $\frac{1}{D_n} \sum_{k\leq n}d_k f(X_k)\rightarrow \int f\,dG$ for all MLR sequences and all bounded and continuous $f$.
			\item \label{cond:second} $\frac{1}{D_n} \sum_{k\leq n}d_k f(X_k)\rightarrow \int f\,dG$ for all MLR sequences and all bounded Lipschitz-functions $f$.
			\item \label{cond:third} $\frac{1}{D_n} \sum_{k\leq n}d_k I\{X_k\leq x\}\rightarrow G(x)$ for all MLR sequences and all $x\in C_G$.
		\end{enumerate}
	\end{lemma}
	
	\begin{lemma}\label{notes2.25}
		Let $(p_n)$ and $(q_n)$ be two positive sequences and let $P_n=\sum_{k\leq n} p_n$, $Q_n=\sum_{k\leq n} q_n$. Then if $p_n/q_n\rightarrow 1$, we have
		\[  \frac{1}{P_n}\sum_{k=1}^n p_ks_k\rightarrow s\qquad \Leftrightarrow \qquad  \frac{1}{Q_n}\sum_{k=1}^n q_ks_k\rightarrow s \]
		for any bounded sequence $(s_n)$ and $s\in \R$.
	\end{lemma}
	
	\begin{lemma}\label{notes3.2}
		Let $(d_k)$ and $(D_n)$ be as in Lemma \ref{portmanteau} and assume that
		\[D_n\rightarrow \infty \quad \text{ and }\quad D_{n+1}/D_n\rightarrow 1  \]
		effectively as $n\rightarrow \infty$, then for each computable $a>1$ there exists a computable subsequence $n_k$ such that 
		\[ D_{n_k}\geq a^k  \quad \text{and}\quad D_{n_k}/a^k \rightarrow 1. \]
	\end{lemma}
	Now, we start with the actual proof. By Lemma \ref{notes2.25}, it is sufficient to show that 
	\[ \lim_{n\rightarrow\infty }\frac{1}{\log (n+1)}\sum_{k=1}^n\log\left(1+\frac{1}{k}\right) I\left\{\frac{S_k}{\sqrt{k}}\leq x \right\}\rightarrow \Phi(x)\quad \text{ for all }x\in \R \]
	holds for all MLR numbers. Letting $d_k=log\left(1+\frac{1}{k}\right), D_n=\sum_{k\leq n}d_k$ and $Y_k=S_k/\sqrt{k}$, the LHS can be rewritten as $\lim_{n\rightarrow\infty }\frac{1}{D_n}\sum_{k=1}^nd_k I\left\{Y_k\leq x \right\}$. By Lemma \ref{portmanteau}, this is equivalent to
	\begin{align}\label{toshowTn}
		\frac{1}{D_n}\sum_{k\leq n}d_k f\left(\frac{S_k}{\sqrt{k}} \right)\rightarrow \int f\,d \Phi
	\end{align}
	for all MLR sequences and all bounded, Lipschitz continuous $f$. It follows from $Y_k \xrightarrow{D} N(0,1)$ and Portmanteau's lemma that
	\[ Ef(Y_k)\rightarrow \int f\,d\Phi \quad \text{as } k\rightarrow\infty \]
	for such functions $f$. It is easy to verify that this implies
	\[ \frac{1}{D_n}\sum_{k=1}^{n}d_k Ef(Y_k)\rightarrow \int f\,d\Phi \quad \text{as } k\rightarrow\infty. \]
	It therefore remains to prove that
	\begin{align}\label{defTn}
		T_n=\frac{1}{D_n}\sum_{k=1}^{n}d_k\left(f(Y_k)-Ef(Y_k) \right)\rightarrow 0 \quad \text{as } n\rightarrow\infty
	\end{align}
	for all MLR numbers. Let
	\[ \xi_k=f(Y_k)-Ef(Y_k), \]
	so that $T_n=\frac{1}{D_n}\sum_{k\leq n}d_k\xi_k$. The main idea of the proof is to show the following lemma.
	\begin{lemma}\label{expTn}
		Let $T_n$ be as defined in \eqref{defTn}, then the following holds.
		\begin{align}\label{Tnbound}
			E[T_n^2]\leq  C \frac{\log \log n}{\log n}.
		\end{align}
		for large $n$ and some $C>0$.
	\end{lemma}
	This lemma is a reformulated version of Proposition 4.3 in \cite{Jonsson} and the proof given at the end of this section is based on the one given in \cite{Jonsson}. First, we show how this result can be used to prove \eqref{ASCLT}.
	
	Let $a>1$ and apply Lemma \ref{notes3.2} with this $a$ and $d_k=\log \frac{k+1}{k}$. Since $D_{n_k}/a^k\geq 1$,
	the inequality above for this subsequence then becomes
	\begin{align}
		ET_{n_k}^2\leq C \frac{k \log a}{a^k}.
	\end{align}
	Hence we have for any $M\in \N$,
	\begin{align*}
		E\sum_{k\geq M}T_{n_k}^2=\sum_{k\geq K} ET_{n_k}^2\leq C \sum_{k\geq M} \frac{k \log a}{a^k}=L_{a,M}<\infty,
	\end{align*}
	with $(a, M)\mapsto L_{a,M}$ computable and $L_{a,M}\rightarrow 0$ as $M\rightarrow \infty$. By Markov's inequality, we have
	\[ \mu\left\{\omega \mid \sup_{k\geq M} T_{n_k}^2>\epsilon \right\}\leq  \mu\left\{\omega \mid \sum_{k\geq M} T_{n_k}^2>\epsilon \right\} = P\left(\sum_{k\geq M} T_{n_k}^2>\epsilon \right)\leq  \frac{E\sum_{k\geq K} T_{n_k}^2}{\epsilon}=\frac{L_{a,M}}{\epsilon}. \]
	Similarly to the proof of Theorem \ref{main_result}, define $K_M^{\epsilon}$ as $\{ \omega \in 2^{\N}\mid \sup\limits_{k\geq M}T_{n_k}^2 >\epsilon \}$ and a function $M(n,l)$ such that $\mu(K^{2^{-l}}_{M(n,l)})<2^{-n-l}$. As explained before, the sets $U_n=\cup_{l=1}^\infty K_{M(n,l)}^{2^{-l}}$ define a Martin-L\"of test (and even a Schnorr test) whose intersection contains all sequences for which $T_{n_k}\nrightarrow 0$. Thus, $T_n\rightarrow 0$ for all Schnorr random sequences. 
	
	Now we show that convergence for the subsequence $T_{n_k}$ implies convergence for the whole sequence $T_n$.
	Consider an arbitrary $n$ and assume that $n_k < n \leq n_{k+1}$. Since $D_{n_k}\sim a^k$, we have
	\[D_{n_{k+1}}/D_{n_k}\rightarrow a. \]
	By assumption there exists an $M$ such that $-M < \xi_k$ (assume $M>0$), for all $k$. Define, for
	$n\geq 1$,
	\[ T_n'=\frac{1}{D_n}\sum_{k\leq n}d_k(\xi_k+M)=T_n+M. \]
	We have that $ T_{n_k}'\rightarrow M $
	for all Schnorr random sequences. Moreover, by positivity
	\[ D_{n_k}T_{n_k}'\leq D_{n}T_{n}'\leq D_{n_{k+1}}T_{n_{k+1}}' \]
	which gives
	\[ \frac{D_{n_k}}{D_{n_{k+1}}}T_{n_k}'\leq  T_n' \leq  \frac{D_{n_{k+1}}}{D_{n_k}}T_{n_{k+1}}'. \]
	By letting $a\rightarrow 1$, we can see that $T_n'\rightarrow M$ and hence $T_n\rightarrow 0$.
	
	~
	
	This finally proves that for all Schnorr random numbers, $T_n\rightarrow 0$ which in turn shows that \eqref{toshowTn} holds, finishing the proof.
	
	~
	
	We end this section with the proof of Lemma \ref{expTn}.
	
	\begin{proof}[Proof of Lemma \ref{expTn}]
		Note that for any numbers $a_k\geq 0$
		\begin{align}\label{easy_ineq2}
			\sum_{k\leq l}a_ka_l \leq 	\left(\sum_k a_k\right)^2\leq 2\sum_{k\leq l}a_ka_l .
		\end{align}
		Applying the second inequality to $T_n$ gives
		\[E[T_n^2]\leq \frac{2}{D_n^2}\sum_{1\leq k\leq l\leq  n}d_kd_lE(\xi_k\xi_l)\leq  \frac{2}{D_n^2}\sum_{1\leq k\leq l\leq  n}d_kd_l|E(\xi_k\xi_l)|.  \]
		For convenience, we define for all $k,l$ with $k\leq l$
		\[ f_{k,l}=f\left( \frac{S_l-S_k}{\sqrt{l}}\right) \]
		Next, let $K$ be the Lipschitz constant of $f$ and assume $|f|\leq K$. Since $f$ is bounded, so is $\xi_k$; say $|\xi_k|\leq C$. Using this, the Lipschitz continuity of $f$ and the fact that $f_{k,l}$ is independent of $\xi_k$, we obtain
		\begin{align*}
			\left|E(\xi_k\xi_l) \right|&=\left|E(\xi_k(\xi_l-f_{k,l})) \right|\\
			&= \left|E \left[ \xi_k  \left( f\left( \frac{S_l}{\sqrt{l}}\right)-E\left[f\left(\frac{S_l}{\sqrt{l}}\right)\right]-f_{k,l} \right) \right]  \right|\\
			&= \left|E\left[\xi_k  \left( f\left( \frac{S_l}{\sqrt{l}}\right)-f_{k,l} \right)  \right]\right|\\
			&\leq C E\left|f\left( \frac{S_l}{\sqrt{l}}\right)-f_{k,l}\right|\\
			&\leq C E\left[   \frac{ K \left|S_k \right|}{\sqrt{l}}  \right]\\
			&=  CK \sqrt{\frac{k}{l}} E\left[  \frac{  \left|S_k \right|}{\sqrt{k}}  \right].
		\end{align*}
		Now, for all $k$ we show that $E\left(\frac{\left|S_k\right|}{\sqrt{k}}\right)\leq 2$. 
		Indeed, let $D=\{ \omega \mid \frac{\left|S_k(\omega)\right|}{\sqrt{k}}<1 \}$, then
		\begin{align*}
			E\left(\frac{\left|S_k\right|}{\sqrt{k}}\right)=E_D\left(\frac{\left|S_k\right|}{\sqrt{k}}\right)+E_{D^C}\left(\frac{\left|S_k\right|}{\sqrt{k}}\right)
			\leq 1+\frac{1}{k}E\left[ S_k \right]^2
			= 1+\frac{1}{k}E\left[\sum_{i\leq k} X_i^2  \right]\leq 2.
		\end{align*}
		With that, we finally obtain
		\[ \left|E(\xi_k\xi_l) \right|\leq C'(k/l)^{1/2}. \]
		
		By the inequality above and using the definition of $D_n$ and $d_n$, 
		\[E[T_n^2]\leq  \frac{2}{D_n^2}\sum_{1\leq k\leq l\leq  n}d_kd_l|E(\xi_k\xi_l)|\leq \frac{C}{(\log(n+1))^2}\sum_{1\leq k\leq l\leq  n}\log\frac{k+1}{k}\log\frac{l+1}{l}\left(\frac{k}{l}\right)^{1/2}.  \]
		To obtain a bound for the sum on the RHS, we define for a fixed $n$ and any $k$ the following sets:
		\begin{align*}
			A_k&=\{l\mid k\leq l\leq n\text{ and }l\leq k\log n \}\\
			B_k&=\{l\mid k\leq l\leq n\text{ and }l> k\log n \}.
		\end{align*}
		Note that for large $k$ ($k>n / \log n$), $A_k=\{ k,k+1,\ldots,n \}$ and $B_k=\emptyset$. 
		\begin{align*}
			E[T_n^2]&\leq  \frac{C}{(\log(n+1))^2}\sum_{1\leq k\leq l\leq  n}\log\frac{k+1}{k}\log\frac{l+1}{l}\left(\frac{k}{l}\right)^{1/2}\\
			&=\frac{C}{(\log(n+1))^2}\left[\sum_{k=1}^n\sum_{l\in A_k }\log\frac{k+1}{k}\log\frac{l+1}{l}\left(\frac{k}{l}\right)^{1/2}+\sum_{k=1}^n\sum_{l\in B_k }\log\frac{k+1}{k}\log\frac{l+1}{l}\left(\frac{k}{l}\right)^{1/2}\right]
		\end{align*}
		Intuitively, we show that if $l\in A_k$ then $\log \frac{l+1}{l}$ is small, and if $l\in B_k$ then $k/l$ is small. 
		
		Starting with $A_k$, note that $A_k=\{k,k+1,\ldots, n_k \}$ for some $n_k\leq n$. Hence
		\[ \sum_{l\in A_k}\log \frac{l+1}{l}=\log \frac{n_k+1}{k}=\log \frac{n_k+1}{n_k}+\log \frac{n_k}{k}\leq \log 2 +\log\log n, \]
		where we used the definition of $A_k$. For large $n$, the above is clearly bounded by $2\log\log n$. Hence
		\begin{align*}
			\sum_{k=1}^n\sum_{l\in A_k }\log\frac{k+1}{k}\log\frac{l+1}{l}\left(\frac{k}{l}\right)^{1/2}&\leq \sum_{k=1}^n\sum_{l\in A_k }\log\frac{k+1}{k}\log\frac{l+1}{l}\\
			&\leq \sum_{k=1}^n\log\frac{k+1}{k}\sum_{l\in A_k }\log\frac{l+1}{l}\leq 2\log (n+1 )\log \log n.
		\end{align*}
		For the indices $l\in B_k$, $k/l<1/ \log n$. Combining this with \eqref{easy_ineq2} where $d_k=\log \frac{k+1}{k}$ gives
		\begin{align*}
			\sum_{k=1}^n\sum_{l\in B_k }\log\frac{k+1}{k}\log\frac{l+1}{l}\left(\frac{k}{l}\right)^{1/2}&\leq \frac{1}{(\log n)^{1/2}}\sum_{k=1}^n\sum_{l\in B_k }\log\frac{k+1}{k}\log\frac{l+1}{l}\\
			&\leq \frac{1}{(\log n)^{1/2}}\sum_{1\leq k\leq l\leq n} \log\frac{k+1}{k}\log\frac{l+1}{l}\\
			&\leq \frac{(\log n+1)^2}{(\log n)^{1/2}}.
		\end{align*}
		Combining the two expressions, we finally obtain for some $C>0$
		\begin{align*}
			E[T_n^2]&=\frac{C}{(\log(n+1))^2}\left[\sum_{k=1}^n\sum_{l\in A_k }\log\frac{k+1}{k}\log\frac{l+1}{l}\left(\frac{k}{l}\right)^{1/2}+\sum_{k=1}^n\sum_{l\in B_k }\log\frac{k+1}{k}\log\frac{l+1}{l}\left(\frac{k}{l}\right)^{1/2}\right]\\
			&\leq C\left[\frac{\log \log n}{\log n+1} + \frac{1}{(\log n)^{1/2}} \right]\leq C' \frac{\log \log n}{\log n},
		\end{align*}
		which completes the proof.
	\end{proof}
	
	~
	
	\textbf{Acknowledgement} We would like to thank Edward De Brouwer for his contributions in the early parts of the research leading up to this paper. 
	
	~
	
	This research received funding from:
	\begin{itemize}[noitemsep]
		\item Research Council KU Leuven: C14/18/092 SymBioSys3; CELSA-HIDUCTION CELSA/17/032
		\item Flemish Government:
		\begin{itemize}[noitemsep]
			\item IWT: Exaptation, PhD grants
			\item FWO 06260 (Iterative and multi-level methods for Bayesian multirelational factorization with features); Elixir I002819N
			\item This research received funding from the Flemish Government (AI Research Program).
			\item VLAIO PM: Augmanting Therapeutic Effectiveness through Novel Analytics
		\end{itemize}
		\item EU: "MELLODDY" This project has received funding from the Innovative Medicines Initiative 2 Joint Undertaking under grant agreement No 831472. This Joint Undertaking receives support from the European Union’s Horizon 2020 research and innovation programme and EFPIA.
	\end{itemize}

\end{document}